\documentclass[a4paper,12pt]{amsart}
\usepackage{amssymb}
\usepackage{ifthen}
\usepackage{graphicx}
\usepackage{color,xcolor}
\usepackage{mathrsfs}
\usepackage{booktabs}
\numberwithin{equation}{section}
\newtheorem{theorem}{Theorem}
\newtheorem{lemma}{Lemma}
\newtheorem{corollary}{Corollary}
\newtheorem{defn}{Definition}
\newtheorem{example}{Example}

\newtheorem{claim}{Claim}

\newtheorem{conj}{Conjecture}
\newtheorem{prob}{Problem}
\newtheorem{remark}{Remark}

\newenvironment{pf}[1][]{%
 \vskip 1mm
 \noindent
 \ifthenelse{\equal{#1}{}}%
  {{\slshape Proof. }}%
  {{\slshape #1.} }%
 }%
{\qed\medskip}

\newcounter{alphabet}
\newcounter{tmp}

\makeatletter
\newcommand{\Ref}[1]{\@ifundefined{r@#1}{}{\setcounter{tmp}{\ref{#1}}\Alph{tmp}}}
\makeatother

\newcommand{\ID}{{\mathbb D}}

\def\be{\begin{equation}}
\def\ee{\end{equation}}

\newcommand{\bee}{\begin{enumerate}}
\newcommand{\eee}{\end{enumerate}}

\newcommand{\blem}{\begin{lem}}
\newcommand{\elem}{\end{lem}}
\newcommand{\bthm}{\begin{thm}}
\newcommand{\ethm}{\end{thm}}
\newcommand{\bcor}{\begin{cor}}
\newcommand{\ecor}{\end{cor}}
\newcommand{\beg}{\begin{example}}
\newcommand{\eeg}{\end{example}}
\newcommand{\begs}{\begin{examples}}
\newcommand{\eegs}{\end{examples}}
\newcommand{\bdefe}{\begin{defn}}
\newcommand{\edefe}{\end{defn}}
\newcommand{\bprob}{\begin{prob}}
\newcommand{\eprob}{\end{prob}}
\newcommand{\bques}{\begin{ques}}
\newcommand{\eques}{\end{ques}}
\newcommand{\bei}{\begin{itemize}}
\newcommand{\eei}{\end{itemize}}

\newcommand{\bde}{\begin{deter}}
\newcommand{\ede}{\end{deter}}
\newcommand{\bca}{\begin{case}}
\newcommand{\eca}{\end{case}}
\newcommand{\bcl}{\begin{claim}}
\newcommand{\ecl}{\end{claim}}
\newcommand{\bcon}{\begin{conj}}
\newcommand{\econ}{\end{conj}}
\newcommand{\bcons}{\begin{conjs}}
\newcommand{\econs}{\end{conjs}}
\newcommand{\bprop}{\begin{propo}}
\newcommand{\eprop}{\end{propo}}
\newcommand{\br}{\begin{rem}}
\newcommand{\er}{\end{rem}}
\newcommand{\brs}{\begin{rems}}
\newcommand{\ers}{\end{rems}}
\newcommand{\bo}{\begin{obser}}
\newcommand{\eo}{\end{obser}}
\newcommand{\bos}{\begin{obsers}}
\newcommand{\eos}{\end{obsers}}
\newcommand{\bpf}{\begin{pf}}
\newcommand{\epf}{\end{pf}}
\newcommand{\ba}{\begin{array}}
\newcommand{\ea}{\end{array}}
\newcommand{\beq}{\begin{eqnarray}}
\newcommand{\beqq}{\begin{eqnarray*}}
\newcommand{\eeq}{\end{eqnarray}}
\newcommand{\eeqq}{\end{eqnarray*}}

\newcounter{minutes}
\divide\time by 60
\newcounter{hours}
\multiply\time by 60 \addtocounter{minutes}{-\time}

\begin{document}
\title[The spherical metric and univalent harmonic mappings]{The spherical
metric and univalent harmonic mappings}

\thanks{
File:~\jobname .tex,
          printed: \number\day-\number\month-\number\year,
          \thehours.\ifnum\theminutes<10{0}\fi\theminutes}

\author[Y. Abu Muhanna]{Yusuf Abu Muhanna}
\address{Y. Abu Muhanna, Department of Mathematics, American University of
Sharjah, UAE-26666.}
\email{ymuhanna@aus.edu}

\author[R. M. Ali]{Rosihan M. Ali}
\address{R. M. Ali, School of Mathematical Sciences, Universiti Sains Malaysia, 11800
USM Penang, Malaysia.}
\email{rosihan@usm.my}

\author[S. Ponnusamy]{Saminathan Ponnusamy
}
\address{S. Ponnusamy, Department of Mathematics,
Indian Institute of Technology Madras, Chennai-600 036, India.
}

\email{samy@isichennai.res.in, samy@iitm.ac.in}


\keywords{Harmonic univalent map, subordination, spherical area, hyperbolic
metric, hyperbolic domain, modular function.
\\
{\tt The article is published in Monatshefte f\"{u}r Mathematik, and in this version minor 
corrections are carried out.
}
}

\begin{abstract}
Let $f=h+\overline{g}$ be a harmonic univalent map in the unit disk
$\mathbb{D}$, where $h $ and $g$ are analytic. We obtain an improved estimate for the second coefficient
of $h$. This indeed is the first qualitative improvement after the appearance of the
papers by Clunie and Sheil-Small in 1984, and by Sheil-Small in 1990.
Also, when the sup-norm of the dilatation is
less than $1$, it is shown that the spherical area of the covering surface
of $h$ is dominated by the spherical area of the covering surface of $f.$
\end{abstract}

\maketitle \pagestyle{myheadings}
\markboth{Y. Abu Muhanna,  R. M. Ali, and S. Ponnusamy}{Spherical metric and harmonic mappings}

\section{Introduction and Preliminaries}\label{AAP1Sec1}

The famous Bieberbach conjecture of 1916 relates to the class $\mathcal S$ of normalized univalent analytic functions
$f(z)=z+\sum_{n=2}^{\infty }a_{n}z^{n}$ defined on the open unit
disk $\mathbb{D}=\{z\in \mathbb{C}:\,|z|<1\}$. The conjecture asserts that $|a_n|\leq n$ for every $f\in {\mathcal S}$ and every $n\geq 2$.
In 1984, de Branges proved this conjecture as well as some other stronger conjectures. Bieberbach's coefficient conjecture
was instrumental in the development of the theory of univalent functions. Numerous methods evolved and applied to investigate
a number of extremal problems in geometric function theory. Yet there still exist many open problems and conjectures involving both univalent and non-univalent mappings. The Keobe
function $k(z)=1/(1-z)^2$ and its rotations $e^{-i\theta}k(e^{-i\theta}z)$ provide solutions to many extremal problems in the
class $\mathcal S$ and related geometric subclasses. These include the class of functions that are close-to-convex, starlike, or convex in some direction (see \cite{PD1,ZN1,CP}).

Another active topic studied in recent years is on planar harmonic mappings (see for instance \cite{BshHeng05,Clunie-Small-84,PD2} and the mini survey \cite{SaRa2013}). The present paper investigates such mappings. Specifically, we treat the
family ${\mathcal{S}}_{H}$ of normalized univalent, sense-preserving
harmonic mappings $f=h+\overline{g}$ in $\mathbb{D}$, where
\begin{equation}\label{harm}
f(z)=z+\sum_{n=2}^{\infty}a_{n}z^{n}+\overline{\sum_{n=1}^{\infty}b_{n}z^{n}}.
\end{equation}
Here a mapping $f=h+\overline{g}$ is sense-preserving if the Jacobian $J_{f}(z)= |h'(z)|^{2} -|g'(z)|^{2}$ of
$f$ is positive in ${\mathbb{D}}$. Set
$${\mathcal{S}}_{H}^{0}=\{f=h+\overline{g}\in {\mathcal{S}}_{H}:\, b_1=g'(0)=0\}
$$
so that each $f\in{\mathcal{S}}_{H}^{0}$ has the form \eqref{harm} with $b_1=g'(0)=0$.
An important member of this family is the so-called harmonic Koebe
mapping $K$ given by
\begin{equation}  \label{harm-koebe1}
K(z)=H(z)+\overline{G(z)}=\mathrm{Re}\left( \frac{z+\frac{1}{3}{z}
^{3}}{(1-z)^{3}}\right) +i\mathrm{Im}\left( \frac{z}{(1-z)^{2}}\right) .
\end{equation}
The classes ${\mathcal{S}}_{H}$ and ${\mathcal{S}}_{H}^{0}$ are known to be
normal \cite{PD2} with respect to the topology of uniform convergence on
compact subsets of ${\mathbb{D}}$. However only ${\mathcal{S}}_{H}^{0}$ is
compact.  In 1984, as a generalization of Bieberbach conjecture, Clunie and Sheil-Small \cite{Clunie-Small-84}
investigated the class ${\mathcal{S}}_{H}^{0}$ and conjectured that if
$f=h+\overline{g}\in{\mathcal{S}}_{H}^{0}$ is given by \eqref{harm}, then for all $n\geq2$,
\begin{equation*}
|a_{n}|\leq\frac{(n+1)(2n+1)}{6},~|b_{n}|\leq\frac{(n-1)(2n-1)}{6} ~\text{and%
} ~\big||a_{n}|-|b_{n}|\big|\leq n
\end{equation*}
with equality occurring for $f(z)=K(z)$ given by \eqref{harm-koebe1}. This conjecture has been verified for a few subclasses of
$\mathcal{S}_{H}^{0}$, namely the class of all functions starlike, close-to-convex, typically real and convex in one direction, in which
$K$ plays the role of extremal function in these subfamilies. It is surprising that the sharp bound for $|a_2|$,
$f=h+\overline{g}\in{\mathcal{S}}_{H}^{0}$, remains unsolved.

In \cite{Clunie-Small-84}, it was shown that $|a_{2}|<12172,$ which later in
\cite{Sheil90} was improved to $|a_{2}|<57$. In \cite{PD2},
the estimate $|a_{2}|<49$ was established, which is far from the conjectured bound
$|a_{2}|\leq5/2$. The field has not seen any further improvements on this problem. The right tools have not been found to deal with this problem and hence, the
above coefficient conjecture in the case of harmonic mappings remains elusive, even in the case of $|a_{2}|$. One of our aims is to consider this problem and prove the following result as a consequence of our new approach.

\begin{theorem} \label{AAP1prop3}
If $f=h+\overline{g}\in {\mathcal{S}}_{H}^{0}$, then $|a_{2}|\leq 20.9197$.
\end{theorem}

The proof of Theorem \ref{AAP1prop3} is presented in Section \ref{AAP1Sec2}. It requires several other basic results
which will be discussed in Section \ref{AAP1Sec1a}.

For $f=h+\overline{g}\in{\mathcal{S}}_{H}^{0}$ given by \eqref{harm}, it was shown in \cite{ML} that the analytic part $h$ of
$f$ lies in Hardy spaces $H^{p}$ for some
small $p >0$, namely, $0<p<(2\alpha_0+2)^{-2}$ with $\alpha_0=\sup_{{\mathcal{S}}_{H}}|a_{2}|$. Thus, determining sharp estimate
for $|a_2|$ is an important
problem. On the other hand, while the bound in Theorem \ref{AAP1prop3} may not be sharp, it is indeed a better estimate than the
known upper bound of $48.4$ (see \cite[p.~96--97]{PD2}). We also note that as an attempt to solve the above conjecture,
the following new conjecture was proposed in \cite{PonSai5}.

\begin{conj}
\label{PS5conj2} If $\mathcal{S}^0_H(\mathcal{S}) = \{h+\overline{g} \in
\mathcal{S}^0_H :\, \Phi_\theta =h+e^{i \theta}g \in \mathcal{S}~
\mbox{for
some}~~\theta \in \mathbb{R}\}$, then $\mathcal{S}^0_H = \mathcal{S}^0_H(%
\mathcal{S})$. That is, for every function $f=h+\overline{g} \in \mathcal{S}%
^0_H$, there exists a $\theta \in \mathbb{R}$ such that $\Phi_\theta =h +
e^{i\theta} g \in \mathcal{S}$.
\end{conj}
The present article is organized as follows. Section \ref{AAP1Sec1a} is devoted to establishing key ideas which lead to the
proof of Theorem \ref{AAP1prop3}.
In Section \ref{AAP1Sec2}, normalized conformal maps are studied in
relation to the elliptic modular function $Q$ on $\mathbb{D}$.
We show in Theorem \ref{AAP1th2} and Corollary \ref{AAP1cor2}, a new estimate on the second coefficient is obtained for functions belonging to a
certain class of conformal mappings.    In Section
\ref{AAP1Sec3-area}, specifically in Theorem \ref{AAP1prop2}, we show that for a
$K$--quasiconformal univalent harmonic map $f=h+\overline{g}$ (that is,
$|f_z|\leq \alpha |f_{\overline{z}}|$ a.e. on ${\mathbb{D}}$, where $\alpha =(K-1)/(K+1)$ with $K\geq 1$), the spherical area $A_{s}(h)$ is
dominated by the spherical area $A_{s}(f),$ and that it is finite. Finally, distortion estimates are obtained in Section \ref{AAP1Sec4} for univalent harmonic mappings.

\section{Background for the proof of Theorem \ref{AAP1prop3}\label{AAP1Sec1a}}
A domain in the complex plane $\mathbb{C}$ is said to be hyperbolic if its
complement contains at least two points. Let $\Omega $ be a hyperbolic domain in
$\mathbb{C}$. Then
the planar uniformization theorem assures the existence of a unique
conformal universal covering $f:\,\mathbb{D}\rightarrow \Omega $ with a
prescribed value $f(0) \in \Omega $, and $f^{\prime }(0)>0$. In the sequel,
$f$ is conformal provided $f^{\prime }(z)\neq 0$ in $\mathbb{D}$. When
$\Omega $ is not a simply connected hyperbolic domain, then the universal
covering $f$ cannot be univalent (see, for example, \cite[p. 41]{BM}).
For a hyperbolic domain $\Omega$ in $\mathbb{C}$, let $d(w,\partial \Omega )$
denote the Euclidean distance between $w\in \Omega $ and
the boundary $\partial \Omega$ of $\Omega $.

The well-known principle of subordination defined by Littlewood in \cite{LW}
will be referred to in this article. For two analytic functions $f$ and $g$
in the unit disk $\mathbb{D}$, the function $f$ is \textit{subordinate} to
$g $, written as $f(z)\prec g(z)$ or $f\prec g$, if there exists an analytic
self-map $\varphi$ of $\mathbb{D}$ with $\varphi (0)=0$ satisfying $f=g\circ
\varphi$ (see also \cite{PD1,ZN1}). When $g$ is univalent in
${\mathbb{D}} $, $f\prec g$ if and only if $f({\mathbb{D}})\subset g({\mathbb{D}})$ and
$f(0)=g(0)$.

When $F:\,D\rightarrow \Omega$ is a universal covering satisfying $f(0)=F(0)$, then
$\varphi (z)=F^{-1}(f(z))$ has a branch at $0,$ which by the
Monodromy theorem can be continued to all of $\mathbb{D}$. Hence $f$ is
subordinate to $F$.

For our purpose, we shall consider the modular function $Q$
\begin{equation}\label{Q}
Q(z) =16z\prod\limits_{n=1}^{\infty }\left( \frac{1+z^{2n}}{1-z^{2n-1}}
\right)^{8} =\sum_{n=1}^{\infty}A_{n}z^{n}, \quad z\in\mathbb{D}.
\end{equation}
From the product expansion of $Q$, we see that the  Taylor coefficients
$A_n $ of $Q$ are all non-negative for $n\geq 1$ with $A_1=16$.

Properties of $Q$  have been comprehensively studied by
Nehari in \cite{ZN2} (see also \cite{ZN1}) in which the author used the
notation $J(z):=-Q(-z)$. Moreover, the fact that the coefficients $\{A_{n}\}$ of $Q$ are nonnegative and form a
non-decreasing convex sequence leads to the following useful result.

\begin{lemma} \label{lem1}
Suppose that $Q(z)=\sum_{n=1}^{\infty }A_{n}z^{n}$ is given by \eqref{Q} and $f(z)=\sum_{n=1}^{\infty }a_{n}z^{n}$ analytic in
$\mathbb{D}$ satisfy $f(z)\prec Q(z)$ for $z\in {\mathbb{D}}$.
Then $|a_{n}|\leq A_{n}$ for $n\geq 1$.
\end{lemma}
\begin{proof}
The assertion follows from \cite{LW, ZN2, Rogo43}. Indeed, it is known from \cite[p. 82]{ZN2} that $\{A_{n}\}$ is a convex
non-decreasing sequence, that is, $B_n=A_n-A_{n-1}$ and $C_n=B_n-B_{n-1}$ are
non-negative, where $A_0=0=A_{-1}$. Since $f\prec Q$, it
readily follows from a theorem of Rogosinski \cite{Rogo43} that $|a_{n}|\leq
A_{n}$ for $n\geq 1$ (see also Littlewood \cite[p. 169]{LW}).
\end{proof}

Let ${\mathcal F}$ denote the class of all analytic functions in $\mathbb{D}$ of the form
$f(z)=\sum_{n=1}^{\infty }a_{n}z^{n}$ that assumed the value $0$ only at $0$.
Clearly the elliptic modular function $J$ and $Q(z):=-J(-z)$ belong to $\mathcal{F}$.

\begin{lemma} {\rm \cite{ZN1}}\label{AAP1lem2a}
If $f\in {\mathcal F}$, $D=f(\ID)$ and $a=d(0,\partial D)$, then $f(z)\prec aQ(z).$
\end{lemma}

If $D$ is a region in the complex plane, denote by $D ^{c}$ its complement $\overline{\mathbb{C}} \backslash D$.
As an immediate consequence of Lemma \ref{AAP1lem2a}, here is a result which reveals an important geometric fact.
%
%

\begin{corollary}\label{AAP1cor1}
\textrm{(Compare with \cite[Theorem II]{ZN2})}
Suppose that  $h\in {\mathcal F}$, $h(z)=\sum_{n=1}^{\infty }a_{n}z^{n}$, $h(\mathbb{D})=D$,
$d(0,\partial D)=|a|$, $a\in \partial D$, and $Q(z)=\sum_{n=1}^{\infty
}A_{n}z^{n}$ is given by \eqref{Q}. Then $|a_{n}|\leq |a|A_{n}$ for $n\geq 1$.
\end{corollary}

An important question to ask is whether the coefficient estimate is sharp. At least in Theorem \ref{AAP1th2} in the next section, we present
better estimates for $a_{2}$ and $a_{3}$.

\section{Coefficient estimates for hyperbolic conformal maps}\label{AAP1Sec2}

Here is our first basic result which gives better estimates for $a_{2}$ and $a_{3}$ than the estimates given by
Corollary \ref{AAP1cor1}.

\begin{theorem} \label{AAP1th2}
Suppose that $h\in {\mathcal F}$, $h(z)=z+\sum_{n=2}^{\infty }a_{n}z^{n}$, $h(\mathbb{D})=D$, and $d(0,\partial D)=|a|$
for some $a\in \partial D$. Then $a_2$ and $a_3$ satisfy the following inequalities:
\begin{equation}
\frac{1}{16}\leq |a|,  \label{bda}
\end{equation}%
\begin{equation*}
|a_{2}|\leq 16|a|+\frac{1}{2|a|},
\end{equation*}%
and
\begin{equation}
|a_{3}|\leq 704|a|.  \label{bda-3}
\end{equation}

If in addition $D$ is hyperbolic, then $|a|<1.$
\end{theorem}
\begin{proof}
For $\rho \in (0,1)$, let
\begin{equation*}
h_{\rho}(z)=(1/\rho)h(\rho z) =z+\sum_{n=2}^{\infty}a_{n}\rho ^{n-1}z^{n}.
\end{equation*}
We now apply Corollary \ref{AAP1cor1} for $h_{\rho}$ with $a_\rho \in \partial
D_\rho$ as the nearest point to the origin, where $D_\rho =h_\rho(\mathbb{D}%
) $. Then it follows from \eqref{Q}, Corollary \ref{AAP1cor1} and \cite[p. 327]{ZN1}
(see also Lemma \ref{AAP1lem2a}) that $h(z)\prec aQ(z)$ and therefore,
\begin{equation*}
h_{\rho}(z)=-\frac{a_\rho}{\rho} Q(\varphi(\rho z))=-\frac{16a_\rho}{\rho}%
[\varphi(\rho z)+8\varphi^{2}(\rho z)+44\varphi ^{3}(\rho z)+\cdots],
\end{equation*}
where $\varphi $ is analytic in $\mathbb{D}$ with $\varphi (0)=0$, $|\varphi (z)|<1$ in $\mathbb{D}$
and $a_\rho \to a$ as $\rho \to 1^{-}$. With
$\varphi(z)=\beta_{1}z+\beta_{2}z^{2}+\cdots$, and comparing the coefficients of $z^n$ for
$n=1,2,3$ in the last expression of $h_{\rho}$, gives the following three relations:
\begin{align*}
a_{1} & =-16a_\rho\beta_{1}=1 \\
a_{2} & =-16a_\rho (\beta_{2}+8\beta_{1}^{2}),~\mbox{ and } \\
a_{3} & =-16a_\rho (\beta_{3}+16\beta_{1}\beta_{2}+44\beta_{1}^{3}).
\end{align*}
Thus, the known estimates $|\beta_{n}|\leq1$ for $n\geq1$ readily
establish
\begin{equation*}
|a_{\rho}|\geq \frac{1}{16|\beta_{1}|}\geq \frac{1}{16},
\end{equation*}
\begin{equation*}
|\beta_{1}|=1/(16|a_{\rho}|)~\mbox{ and }~|a_{2}|\leq16|a_{\rho
}|(1+8|\beta_{1}|^{2})=16|a_{\rho}|+\frac{1}{2|a_{\rho}|}.
\end{equation*}
Note that $|a_{\rho}|\rightarrow |a|$ as $\rho\rightarrow 1^{-}.$ Clearly, the
right side of \eqref{bda} follows from the analog of Koebe one-quarter
theorem for hyperbolic domains (\cite{BP} and \cite[p. 894]{SV}). See also
the inequality recalled in \eqref{euchypdom}.

Finally, we present a proof of \eqref{bda-3}. To do this, we recall the
sharp upper bounds for the functionals $\left|\beta_3+\mu \beta_1\beta_2+\nu
\beta_1^3\right|$ when $\mu$ and $\nu $ are real. In \cite{PS}, Prokhorov
and Szynal proved among other results that
\begin{equation*}
\left|\beta_3+\mu \beta_1\beta_2+\nu \beta_1^3\right|\leq |\nu|
\end{equation*}
if $|\mu|\ge 4$ and $\nu \geq (2/3)(|\mu| -1)$. From the third relation
above for $a_3$, this condition is fulfilled (since $\mu =16$ and $\nu =44$)
and thus,
\begin{equation*}
|a_{3}| = 16|a_\rho|\, \left
|\beta_{3}+16\beta_{1}\beta_{2}+44\beta_{1}^{3}\right | \leq 16 \times 44\,
|a_\rho| =704 |a_\rho|
\end{equation*}
which proves the desired inequality \eqref{bda-3}.
\end{proof}

As $A(x)=16x +1/(2x)$ is increasing on $[1/(4\sqrt{2}),1]$, it follows that $%
A(x)\leq A(1)=16.5$. This observation leads to

\begin{corollary}
\label{AAP1cor2} If  $h\in {\mathcal F}$ and $h(z)=z+\sum_{n=2}^{\infty }a_{n}z^{n}$, then $|a_{2}|\leq 16.5$ and $|a_{3}|\leq 704$.
\end{corollary}

Now, we are in a position to formulate an important general result.
First note that if $h(z)=\sum_{n=1}^{\infty}a_{n}z^{n}$ is analytic on $\overline{{\mathbb{D}}}
$, $h({\mathbb{D}})=D$ and $d(0,\partial D)=|a|$ for some $a\in \partial D$, then the function
$z\left( h(z)-a\right) /(-a)$ belongs to ${\mathcal{F}}$. Furthermore, we remark that $z(h(z)-a)$ is zero only at $0$.

\begin{theorem}\label{AAP1th3a}
Suppose that $h$ is conformal in ${\mathbb{D}}$, $h({\mathbb{D}})=D$ is hyperbolic
and $d(0,\partial D)=|a|$, where $a\in \partial D$ and $h(0)=h'(0)-1=0$. 
Then
\begin{equation*}
\frac{1}{16.5}\leq |a|<1.
\end{equation*}
\end{theorem}
\begin{proof}
Let
\begin{equation*}
g(z)=z\frac{a-h(z)}{a}=z-\frac{1}{a}z^{2} +\cdots,
\end{equation*}%
so that $g$ belongs to $\mathcal{F}$. Consequently, by Theorem \ref{AAP1th2}, it follows that $|1/a|\leq 16.5$.
Since $D$ has a hyperbolic metric $\lambda (z)$ and that
\begin{equation*}
\lambda (0)d(0,\partial D)=\frac{1}{h^{\prime }(0)}d(0,\partial D)=|a|<1,
\end{equation*}%
the result follows.  At this place it is worth recalling that $\lambda (z)d(z,\partial D)\leq 1$ holds always.
\end{proof}

For the proof of our next lemma, we need to establish some preliminaries. It is known from
the work of Abu Muhanna and Hallenbeck \cite{AbHal-08} that if $\alpha >0,$
and
\begin{equation*}
E_{\alpha }=\left\{f\in {\mathcal{A}}:\, f(z)\prec \exp (\alpha z/(1-z)),~
f(0)=1\right\},
\end{equation*}
then
\begin{equation*}
E_{\alpha }=\left\{ \int _{\partial {\mathbb{D}}} \exp \left ( \alpha \frac{%
xz}{1-xz}\right )d\mu (x): \,\mu \mbox{ is a probability
measure on $\mathbb{D}$}\right\} .
\end{equation*}
Each function $f\in E_{\alpha }$ maps ${\mathbb{D}}$ into $|w|>r=\exp
(-\alpha /2)$ and $f(0)=1.$ Clearly, the inclusion $E_{\alpha }\subset
E_{\beta }$ holds for $\alpha >\beta $ and thus, the coefficients of $f\in
E_\alpha $ are dominated by the corresponding coefficients of $F(z)=\exp
(\alpha z/(1-z))$, where
\begin{equation*}
F(z)=\exp (\alpha z/(1-z))=1+\alpha z+\frac{\alpha (\alpha +2)}{2}
z^{2}+\cdots =1+A_{1}z+A_{2}z^{2}+ \cdots .
\end{equation*}
In particular,
\begin{equation*}
\left |\frac{f^{\prime \prime }(0)}{2!}\right | \leq A_2=\frac{\alpha (\alpha +2)}{2} .
\end{equation*}
Clearly, as $\alpha \rightarrow \infty $ (circle $|w|=\rho >r=\exp (-\alpha /2)$ shrinks)  $r\rightarrow 0$
and $A_{1},A_{2}\rightarrow \infty .$

\begin{lemma}\label{AAP1lem2}
Suppose that $h(z)=z+a_{2}z^{2}+\cdots $ is analytic in
$\overline{\mathbb{D}}$ and misses the disk ${\mathbb{D}}(c,r):=\{z:\,|z-c|<r\}$ which touches
the boundary $\partial h({\mathbb{D}})$. Then the function $\Psi$ defined by $\Psi(z) =\frac{c-h(z)}{c}$
misses the disk ${\mathbb{D}}(0,\rho )$, where $\rho =r/|c|>1/16$,
and $|a_{2}|<20.9197|c|.$
\end{lemma}
\begin{proof}
By assumption, the function
\begin{equation*}
\Psi(z)=\frac{c-h(z)}{c}=1-\frac{1}{c}z-\frac{a_{2}}{c}z^{2}+\cdots
\end{equation*}%
misses the disk ${\mathbb{D}}(0,r/|c|)$ and its boundary touches the circle $%
|w|=r/|c|.$ Then $g$ defined by $g(z)=z\Psi(z)$ 
belongs to the family $\mathcal{F}$ and thus, if the nearest point to $0$ is $%
a=g(e^{i\theta })$, then $|a|=r/|c|\geq 1/16$.  Note that $z(c-h(z))$ is zero only at $0$
and thus, $\rho >1/16$ is indeed a consequence of Nehari's result.

Consequently,
\begin{equation*}
\left\vert \frac{a_{2}}{c}\right\vert \leq \frac{\alpha (\alpha +2)}{2},
\end{equation*}%
where
\begin{equation*}
\frac{r}{|c|}=\exp (-\alpha /2).
\end{equation*}
Thus
$$\alpha =\log \left(
|c|/r\right) ^{2}<2\log \left( 16\right) =8\log 2 \approx 5.54518.$$
This gives the estimate
\begin{equation*}
\left\vert a_{2}\right\vert \leq |c|\frac{\alpha (\alpha +2)}{2}<|c|[8\log
(2)(4\log 2+1)]\approx 20.9197 |c|
\end{equation*}%
and completes the proof.
\end{proof}

The proof of Theorem \ref{AAP1prop3} below depends on the following remark.

\begin{remark}
Suppose that $h(z)=z+a_{2}z^{2}+\cdots $ is conformal on $\overline{{\mathbb{
D}}}$ and $d(0,\partial h({\mathbb{D}}))=|a|<1$, where $a\in \partial h({\mathbb{D}}).$
It is worth pointing out that ${\mathbb{D}}(a,1-|a|)\cap (h({\mathbb{D}}))^c $ is non-empty and open,
where ${\mathbb{D}}(a,r):=\{z:\,|z-a|<r\}$. Thus, there is a
complex number $c$ in the complement $(h({\mathbb{D}}))^c$ with $|c|<1-|a|$ and a positive number $r$ so that
\begin{equation*}
{\mathbb{D}}(c,r)\subset (h({\mathbb{D}}))^c
\end{equation*}%
and touches the boundary $\partial h({\mathbb{D}}).$ With this $c,$
$|a_{2}|<20.9197.$
\end{remark}

In this remark, it suffices to assume that  $h(z)=z+a_{2}z^{2}+\cdots $ is conformal on $\ID$; otherwise,  consider
$h_{\rho}(z)=(1/\rho)h(\rho z)$, $\rho \in (0,1)$ in the proof, and then let $\rho \to 1^{-}$.

Although $h^{\prime}(z)\neq0$ for $f=h+\overline{g}\in{\mathcal{S}}_{H}^{0}$, the function $h$ can however vanish several times.
This fact is illustrated by the first author in \cite{Mu} by the function
\begin{equation*}
h(z)=\frac{-\coth z+z+(\log\sinh z-\log\sinh1)+\coth1-1}{2}, ~z\in R.
\end{equation*}
Here $R$ is the open right half-plane and the dilatation $\varphi(z)=(\coth
z-1)/(\coth z+1)$ maps $R$ onto the punctured disk $\mathbb{D}%
\setminus\{0\}. $ In \cite{Mu}, the function $h$ was shown to have infinite
valence and finitely many zeros.

Finally, we conclude the section with the proof of Theorem \ref{AAP1prop3}.

\vspace{6pt}
\noindent
{\bf Proof of Theorem \ref{AAP1prop3}.}
Let $\rho _{n}$ be a sequence of radii increasing to $1$.
Now, consider the analytic part of $f_{\rho _{n}}(z)=(1/\rho _{n})f(\rho
_{n}z)$, namely, the function
\begin{equation*}
h_{\rho _{n}}(z)=(1/\rho _{n})h(\rho _{n}z)=z+\sum_{k=2}^{\infty }a_{k}\rho
_{n}^{k-1}z^{k}.
\end{equation*}
Since $D_{\rho _{n}}=h_{\rho _{n}}(\mathbb{D})$ is hyperbolic, by Lemma \ref%
{AAP1lem2}, the second coefficient of $h_{\rho _{n}}$ gives the estimate $%
|a_{2}\rho _{n}|\leq 20.9197$ for each $n$. The desired conclusion follows
when $n\rightarrow \infty $.
\hfill $\Box$

%
%
%

\section{Spherical area of the covering surface of $h$ over $D.$ \label%
{AAP1Sec3-area}}

The spherical metric on $\overline{\mathbb{C}}$ (the Riemann sphere) is
defined by
\begin{equation*} 
\sigma (z)|dz|=\frac{|dz|}{1+|z|^{2}},
\end{equation*}
and the spherical area of (a surface of $f$ above) $\mathbb{D}$ given by a harmonic
map $f$ is
\begin{equation*}
A_{s}(f)=\displaystyle\iint \limits_{{\mathbb{D}}}\frac{J_{f}(z)\,dA}{\left(
1+|f(z)|^{2}\right) ^{2}},
\end{equation*}
where $J_f$ denotes the Jacobian of $f$. When $f$ is analytic, then the
spherical area becomes
\begin{equation*}
A_{s}(f)=\displaystyle\iint \limits_{{\mathbb{D}}}\frac{|f^{\prime
}(z)|^{2}\,dA}{\left( 1+|f(z)|^{2}\right) ^{2}}.
\end{equation*}
Clearly, if the surface covers the plane exactly once then $A_{s}(f)=4\pi .$

The hyperbolic (or Poincar\'e) metric in $\mathbb{D}$ \cite{BM,SA, GL,SV} is
the Riemannian metric defined by $\lambda _{\mathbb{D}}(z)|dz|$, where
\begin{equation*}
\lambda _{\mathbb{D}}(z)=\frac{1}{1-|z|^{2}}.
\end{equation*}
Note that metrics $\sigma $ and $\lambda _{\mathbb{D}}$ have constant
curvatures $4$ and $-4$, respectively. Using analytic maps, hyperbolic
metrics can be transferred from one domain to another. Indeed, for a
given hyperbolic domain $\Omega $, and a (conformal) universal covering map $%
f:\,\mathbb{D} \rightarrow \Omega $, the hyperbolic metric of $\Omega $ is
given by
\begin{equation*}
\lambda _{\Omega }(f(z))|f^{\prime }(z)|\,|dz|=\lambda _{\mathbb{D}}(z)|dz|.
\end{equation*}
In particular,
$$\lambda _{\Omega }(f(0))=\frac{1}{|f'(0)|},
$$
a fact which is already used in the proof of Theorem \ref{AAP1th3a}.
It is well-known that the metric $\lambda _{\Omega }$ is independent of the
choice of the conformal map $f$ used.

When $\Omega $ is simply connected, the Koebe one-quarter theorem \cite[p. 22%
]{CP} gives the sharp estimates
\begin{equation}
\frac{1}{4}\leq d(w,\partial \Omega )\lambda _{\Omega }(w)\leq 1.
\label{Koe}
\end{equation}
When $\Omega $ is a hyperbolic domain, the estimates established in \cite{BP}
and \cite[p. 894]{SV} are
\begin{equation}
\frac{1}{2(\beta _{\Omega }(w)+C_{0})}\leq d(w,\partial \Omega )\lambda
_{\Omega }(w)\leq \min \left\{ 1,\frac{2C_{0}+\pi /2}{2(\beta _{\Omega
}(w)+C_{0})}\right\} ,  \label{euchypdom}
\end{equation}
where
\begin{equation*}
\beta _{\Omega }(w)=\underset{b\in \partial \Omega }{\inf }\left | \log
\left | \frac{w-a}{b-a}\right | \,\right | ,
\end{equation*}
with $a\in \partial \Omega $, $|w-a|=d(w,\partial \Omega )$, and $%
C_{0}\approx 4.37688$. The lower bound in \eqref{euchypdom} is known to be
sharp but not the upper bound. Indeed it is difficult to estimate $\beta
_{\Omega }(w)$, which is the modulus of the largest annulus inside $\Omega $
separating the boundary $\partial \Omega $. However if $\partial \Omega $ is
a connected set, then $\beta _{\Omega }(w)=0$ and a lower bound is $1/8$,
which is the right estimate that one could get from \eqref{euchypdom}.

In the following result, $A_{s}(D):=A_{s}(D,h)$ denotes spherical area of
the covering surface of $h$ over $D.$ Moreover, if  $f=h+\overline{g}$ is a sense-preserving
harmonic mapping in $\mathbb{D}$, then $J_{f}(z)= |h^{\prime}(z)|^{2} -|g^{\prime}(z)|^{2}>0$ in
${\mathbb{D}}$ and thus, the analytic dilatation $\varphi(z):= \varphi_f(z) =g^{\prime}(z)/h^{\prime}(z)$ of $f$
satisfies $|\varphi(z)|<1$ for all $z\in\mathbb{D}$. Define $||\varphi||_{\infty}:=\sup_{z\in\mathbb{D}}|\varphi (z)|.$

\begin{theorem}
\label{AAP1prop2} Let $f=h+\overline{g}$ be a sense-preserving univalent
harmonic mapping in $\mathbb{D}$ given by \eqref{harm} with the dilatation $\varphi$.
Suppose that $\alpha =||\varphi||_{\infty}<1$ and $D=h({\mathbb{D}})$ is
hyperbolic. Then the spherical area of the covering surface of $D$ satisfies
\begin{equation*}
\frac{1-\alpha^{2}}{4} A_{s}(D)\leq A_{s}(\Omega):=A_{s}(\Omega ,f)\leq 4\pi,
\end{equation*}
where $f({\mathbb{D}})=\Omega$ and
\begin{equation*}
A_{s}(D)=\displaystyle\iint\limits_{{\mathbb{D}}}\frac{|h^{\prime}(z)|^{2}%
\,dA}{\left( 1+|h(z)|^{2}\right) ^{2}}.
\end{equation*}
\end{theorem}
\begin{proof}
Let $f=h+\overline{g}$ and $N(a,r)$ be a disk in $D$ so that the components
of $h^{-1}(N(a,r)),$ namely, $\{U_{n}^{a}\}$, are disjoint, open, connected
and $h$ maps each component univalently onto the disk $N(a,r).$ In other
words, $\{U_{n}^{a}\}$ is the covering above $N(a,r)$. Let $U_{n}^{a}$ be an
arbitrary component with $h(b)=a$ and $b\in U_{n}^{a}.$ We now introduce
\begin{equation*}
F_{n}(w)=f(h^{-1}(w))~\mbox{ for }~ w\in N(a,r).
\end{equation*}
Then $F_{n}(w)=w+\overline{g(h^{-1}(w))}$ and therefore, with $D=h({\mathbb{D%
}})$ and $z=h^{-1}(w)$, it follows easily that
\begin{equation*}
(F_{n})_{w}(w)=1 ~\mbox{and} ~\overline{(F_{n})_{\overline{w}} (w)}= \frac{%
g^{\prime}(z)}{h^{\prime}(z)}=\varphi(z)= \varphi(h^{-1}(w)).
\end{equation*}

Since $\{U_{n}^{a}\}$ are disjoint and $f$ is univalent, $%
F_{n}(N(a,r))=f(U_{n}^{a})$ are also disjoint. Consequently, the spherical
area of the image of $F_{n}$ above $U_{n}^{a}$ is given by
\begin{equation}  \label{AAP1eq2}
A_{s}(F_{n})=\displaystyle\iint\limits_{N(a,r)}\frac{J_{F_{n}}(z)\,dA_{w}}{%
\left( 1+|F_{n}(w)|^{2}\right) ^{2}} =\displaystyle\iint \limits_{N(a,r)}%
\frac{(1-|\varphi(h^{-1}(w))|^{2})\,dA_{w} }{\left( 1+|F_{n}(w)|^{2}\right)
^{2}}.
\end{equation}
Note that
\begin{equation*}
|F_{n}(w)-w|\leq\left| g(h^{-1}(w))\right| \leq\displaystyle\int
\limits_{a}^{w}\left| \varphi(h^{-1}(w))\right| |dw|\leq|w-a|
\end{equation*}
and therefore, $|F_{n}(w)|\leq|w|+|w-a|.$ When $a=0$, we get the estimate $%
|F_{n}(w)|\leq2|w|$. On the other hand, when $a\neq0$ we may add the
condition $r<|a|/2$ so that $|w-a|\leq|a|/2\leq|w|,$ which again yields $%
|F_{n}(w)|\leq2|w|.$ In both cases, \eqref{AAP1eq2} gives
\begin{align*}
A_{s}(F_{n}) & \geq \frac{1-\alpha^{2}}{4}\displaystyle \iint%
\limits_{N(a,r)} \frac{4\,dA_{w} }{\left( 1+4|w|^{2}\right) ^{2}} \\
& \geq\frac{1-\alpha^{2}}{4}A_{s}(2N(a,r)),
\end{align*}
where $A_{s}(2N(a,r))$ is the spherical area of $2N(a,r).$ Thus, the
spherical area of the part of the covering surface of $D$ above $N(a,r)$ is
dominated by the spherical area of the union of $\{f(U_{n}^{a})\}.$ In other
words,
\begin{equation*}
A_{s}(N(a,r))\leq\displaystyle\sum\limits_{n}A_{s}(f(U_{n}^{a}))\leq\frac {4%
}{1-\alpha^{2}}A_{s}(\Omega).
\end{equation*}

Let $D_{1}=\bigcup\limits_{j}N(a_{j},r_{j})$ be disjoint union of disks in $%
D,$ and $U_{1}=\bigcup\limits_{j}\bigcup\limits_{n}U_{n}^{a_{j}}$ the
corresponding disjoint covers. Note that $\{U_{n}^{a_{j}}\}$ are disjoint,
for all $a_{j}$ and $n.$ Then the spherical area of the covering surface of $%
D$ above $D_{1}$ is
\begin{align*}
A_{s}(h(U_{1})) & =\displaystyle\sum\limits_{j=1}^{\infty}\displaystyle\sum
\limits_{n}A_{s}(h(U_{n}^{a_{j}})) \\
& \leq\frac{4}{1-\alpha^{2}}\displaystyle\sum\limits_{j=1}^{\infty }%
\displaystyle\sum\limits_{n}A_{s}(f(U_{n}^{a_{j}})) \\
& \leq\frac{4}{1-\alpha^{2}}A_{s}(\Omega) \\
& \leq\frac{16\pi}{1-\alpha^{2}} .
\end{align*}
Consequently, the spherical area of the covering surface of $D$ is less than
or equal to $16\pi/(1-\alpha^{2}),$ which completes the proof.
\end{proof}

We end the section with the following.

\begin{conj}
The condition on the dilatation in Theorem \ref{AAP1prop2} can be removed.
\end{conj}

\section{Distortion estimates for univalent harmonic mappings}

\label{AAP1Sec4}

For $f=h+\overline{g}\in{\mathcal{S}}_{H}$, it is known \cite[ pp. 92, 98]%
{PD2} that
\begin{equation}
\frac{1}{16}(1-|\mu(z)|)\leq d(f(z),\partial\Omega)\lambda_{D}(h(z))\leq c,
\label{eq3.1}
\end{equation}
where $\mu$ is the dilatation of $f$, $1\leq c<2$, $f({\mathbb{D}})=\Omega$,
and $h({\mathbb{D}})=D$. If $D$ is hyperbolic, $h$ satisfies the general
inequality in (\ref{euchypdom}). Moreover, since $D$ is a simply connected
domain, (\ref{Koe}) gives
\begin{equation}
\frac{1}{4}\leq d(f(z),\partial\Omega)\lambda_{\Omega}(f(z))\leq 1.
\label{eq3.2}
\end{equation}
It follows from (\ref{eq3.1}) and (\ref{eq3.2}) that
\begin{equation*}
\frac{1}{16}(1-|\mu(z)|)\lambda_{\Omega}(f(z))\leq\lambda_{D}(h(z))\leq
4c\lambda_{\Omega}(f(z)).
\end{equation*}
Combining (\ref{euchypdom}) and (\ref{eq3.1}) yields
\begin{equation}
\frac{1}{16}(1-|\mu(z)|)\leq\frac{d(f(z),\partial\Omega)}{d(h(z),\partial D)}%
\leq2c(\beta_{D}(w)+C_{0}).  \label{eq3.4}
\end{equation}

The next result gives better estimates than (\ref{eq3.4}).

\begin{theorem}
Let $f=h+\overline{g}\in{\mathcal{S}}_{H}^{0}$, $f({\mathbb{D}})=\Omega$, $h(%
{\mathbb{D}})=D$ is hyperbolic, and $a\in\partial D$ be the nearest point to
the origin $0$. Then
\begin{equation}\label{bdf-1a}
\frac{1}{16} \leq d(0,\partial D)\leq1
\end{equation}
and
\begin{equation}
\frac{1}{16}(1-|\mu(z)|) \leq d(f(z),\partial\Omega)\leq2d(h(z),\partial D).
\label{bdf}
\end{equation}
\end{theorem}
\begin{proof}
The inequalities \eqref{bdf-1a} have already appeared in Theorem \ref{AAP1th2}.
To show the inequalities (\ref{bdf}), we fix $z\in{\mathbb{D}}$ and let
$b\in\partial D$ be nearest to $h(z)$ and $\gamma=[h(z),b)\subset D$ be the
line segment from $h(z)$ to $b$. As $b$ is accessible from inside $D$, $b $
is a radial limit for $h$. The function $h$ being conformal assures a branch
of $\Gamma=h^{-1}(\gamma)$ connecting $z$ to $\partial{\mathbb{D}}$. Now, we
find that
\begin{align*}
d(f(z),\partial\Omega) & \leq\int\limits_{\Gamma} \left\vert \frac{ \partial
f}{\partial\zeta}d\zeta+\frac{\partial f}{\partial\overline{\zeta}}d%
\overline{\zeta}\right\vert \\
& \leq\int\limits_{\Gamma}\left( |h^{\prime}(\zeta)|+|g^{\prime}
(\zeta)|\right) |d\zeta| \\
& \leq2\int\limits_{\Gamma}|h^{\prime}(\zeta)|\, |d\zeta|=2d(h(z),\partial D)
\end{align*}
which proves the right-hand inequality of \eqref{bdf} The left-hand
inequality of (\ref{bdf}) is a known estimate (see \cite{PD2}).
\end{proof}

%
%

\subsection*{Acknowledgment.}

This work was supported by the American University of Sharjah and by a research university grant from Universiti Sains Malaysia.
The third author is currently at Indian Statistical Institute (ISI), Chennai Centre, Chennai, India.
The research of the third author was supported by the project RUS/RFBR/P-163 under
Department of Science \& Technology (India).

\subsection*{Conflict of Interests}
The authors declare that there is no conflict of interests regarding the publication of this paper.

\end{document}